\documentclass[11 pt,twoside]{amsart} 
\usepackage{hyperref,pifont,amssymb,amsmath,amsthm,lipsum,placeins,amsfonts,amsopn, latexsym,graphicx, mathrsfs, verbatim, enumerate, url, multicol,diagrams,color,footnote }
\usepackage[stable]{footmisc}
\usepackage[margin=1in]{geometry} 
\usepackage[stable]{footmisc} 

\def\address#1{\expandafter\def\expandafter\@aabuffer\expandafter
	{\@aabuffer{\affiliationfont{#1}}\relax\par
	\vspace*{13pt}}}

\DeclareMathOperator{\re}{Re}

\DeclareMathOperator{\TId}{TanId}

\begin{document}

\numberwithin{equation}{section}
\newtheorem{thm}{Theorem}[section]  
\newtheorem*{thm*}{Theorem}
\newtheorem*{thmA}{Theorem A}
\newtheorem*{thmB}{Theorem B}
\newtheorem*{propC}{Proposition C}
\newtheorem{thmx}{Theorem}
\newtheorem{sublem}{Lemma}[thm]
\newtheorem{lem}[thm]{Lemma} 
\newtheorem{cor}[thm]{Corollary}
\newtheorem{prop}[thm]{Proposition}
\theoremstyle{remark}
\newtheorem{rem}[thm]{Remark}
\newtheorem{remx}[thmx]{Remark}
\newtheorem{ex}[thm]{Example}
\theoremstyle{definition}
\newtheorem{Def}[thm]{Definition}
\renewenvironment{proof}[1][\proofname]{{\bfseries #1. }}{\qed \\ \newline} 

\newcommand{\BO}{O}
\newcommand{\lo}{o}
\newcommand{\0}{\mbox{O}}
\newcommand{\R}{\mathbb{R} }
\newcommand{\C}{\mathbb{C} }
\newcommand{\Z}{\mathbb{Z}}
\newcommand{\p}{\mathbb{P}}
\newcommand{\ol}{\overline}
\newcommand{\N}{\mathbb{N}}
\newcommand{\wt}{\widetilde}
\newcommand{\Id}{\operatorname{Id}}
\newcommand{\Arg}{\operatorname{Arg}}
\newcommand{\A}{\hat{A}_2}
\newcommand{\B}{\hat{B}_2}
\newcommand{\F}{\mathcal{F}}
\newcommand{\Oa}{\Omega_1}
\newcommand{\Ob}{\Omega_2}
\newcommand{\Oc}{\Omega_3}

\newcommand\blfootnote[1]{
  \begingroup
  \renewcommand\thefootnote{}\footnote{#1}
  \addtocounter{footnote}{-1}
  \endgroup
}
\renewcommand{\thefootnote}{\alph{footnote}}

\title{Domain of attraction for maps tangent to the identity in $\C^2$ with characteristic direction of higher degree}
\author{Sara Lapan}
\date{\today}
\maketitle
\begin{center}Department of Mathematics\\ Northwestern University, \\ Evanston, IL; USA \\
slapan@math.northwestern.edu
\end{center}
\begin{abstract}We study holomorphic fixed point germs in two complex variables that are tangent to the identity and have a degenerate characteristic direction.  We show that if that characteristic direction is also a characteristic direction for higher degree terms, is non-degenerate for a higher degree term, and satisfies some additional properties, then there is a domain of attraction on which points converge to the origin along that direction. \end{abstract}

\pagestyle{myheadings}
\markright{Attracting domain}

\section*{Introduction}
In this paper, we will be studying germs of holomorphic self-maps of $\C^2$ that are tangent to the identity at a fixed point and have a degenerate characteristic direction $[v]$ that satisfies some additional properties.  The conditions of the main theorem allow for $[v]$ to be a degenerate characteristic direction of any type (i.e., dicritical, Fuchsian, irregular, or apparent).  See \S\ref{prelim} for definitions.

\begin{thmA}\label{thmA}
Let $f$ be a germ of a holomorphic self-map of $\C^2$ that is tangent to the identity at a fixed point $p$, is of order $k+1$, and has characteristic direction $[v]\in\p^1(\C)$.  Assume that $[v]$ is:
\begin{enumerate}
\item a characteristic direction of degree $s\leq\infty$;
\item non-degenerate of degree $r+1$, where $k<r<s$; and
\item of order one in degree $t+1$, where $k\leq t\leq r$.
\end{enumerate}
If $[v]$ is transversally attracting and $s>r+t-k$, 
then there exists a domain of attraction whose points, under iteration by $f$, converge to $p$ along $[v]$. 
\end{thmA} 

\begin{remx}Suppose $f$ satisfies conditions (1)-(3).  If $r\not\in\{t,2t\}$, then $[v]$ is transversally attracting.  If $r\in\{t,2t\}$, then $[v]$ might not be transversally attracting.  
In addition, $f^{-1}$ will also satisfy conditions (1)-(3) with the same $k,r,s,t$.  If $r=2t$, then $[v]$ must be transversally attracting for $f$ and/or $f^{-1}$.  See Definition~\ref{transv} and Lemma~\ref{f-1} for more details.
\end{remx}

\begin{remx}For a map that satisfies the conditions of Theorem~\hyperref[thmA]{A}, estimates on the rate at which points in its domain of attraction to $p$ converge are given in Proposition~\ref{size}.\end{remx}


Theorem~\hyperref[thmA]{A} is an extension of the following theorem due to Rong that, among other things, assumes $t=k$ and, as a consequence of its assumptions, the characteristic direction is apparent \cite{Ro2}.
\begin{thm*}[Rong, 2015]\label{Rothm}
Let $f$ be a germ of a holomorphic self-map of $\C^2$ that is  tangent to the identity at a fixed point $p$ and that is of order $k+1$.  Assume that $[v]$ is:
\begin{enumerate}
\item a degenerate characteristic direction of degree $k+1$, 
\item essentially non-degenerate\footnote{Assumptions (1)-(2) imply that $[v]$ must be: a  characteristic direction of degree $s$, of order one in degree $k+1$ (so $t=k$), and non-degenerate of degree $r+1$ for some $k<r<s$.}, and 
\item of order one in degree $k+1$. 
\end{enumerate}
If $[v]$ is transversally attracting, then there exists a domain of attraction whose points, under iteration by $f$, converge to $p$ along $[v]$.
\end{thm*}



The conditions of Theorem~\hyperref[thmA]{A} on $f$ are not invariant under conjugation by every biholomorphism that fixes the origin.  The following two results provide sufficient conditions on a biholomorphism, $\Psi$, to ensure invariance under conjugation by $\Psi$ of the conditions on $f$ given in Theorem~\hyperref[thmA]{A}.  In Theorem~\hyperref[thmB]{B}, we make assumptions on a biholomorphism in $\C^n$ that partially fixes a direction and we see how these assumptions translate to specific properties of the biholomorphism and its inverse.

\begin{thmB}
\label{thmB}
Let $\Psi$ be a biholomorphism of $\C^n$ that fixes the origin, $\0$.  Near the origin:
$$\Psi(z,w)=\sum_{j\geq1} \Psi_j(z,w),$$ 
where $(z,w)\in\C\times\C^{n-1}$ and $\Psi_j$ is homogeneous of degree $j$.  Let $[1:\0]:=[1:0:\cdots:0]\in\p^{n-1}(\C)$.  If $\Psi$ fixes $[v]\in\p^{n-1}(\C)$ up through degree $\sigma$, then:
\begin{enumerate}
\item each $\Psi_j$ fixes $[v]$ for all $j\leq\sigma$;
\item $\Psi^{-1}$ fixes $[v]$ up through degree $\sigma$; and
\item given an invertible linear map $L$ that fixes $\0$ and sends $[1:\0]$ to $[v]$, there is a holomorphism $\phi:\C\to\C^{n-1}$ and a biholomorphism $\chi:\C^{n}\to\C^{n}$ that fixes $\0$ and $[1:\0]$ such that:
$$L^{-1}\circ\Psi\circ L=\Phi\circ \chi,$$
where $\Phi=\Id+(0,z^{\sigma+1}\phi(z))$ is a biholomorphism.
\end{enumerate} 
 \end{thmB} 

In the following proposition, we combine our assumptions on $f$ from Theorem~\hyperref[thmA]{A} with properties of biholomorphisms from Theorem~\hyperref[thmB]{B} to get conditions under which $k,t,r,s$ are invariant.


\begin{propC}
\label{propC} 
Let $f$ be a germ of a holomorphic self-map of $\C^2$ that is tangent to the identity at a fixed point $p$, is of order $k+1$, and satisfies conditions (1)-(3) of Theorem~\hyperref[thmA]{A} with characteristic direction $[v]$.  Suppose that $\Psi$ is a biholomorphism of $\C^2$ that fixes $p$ and, near $p$, fixes $[v]$ up through degree $\sigma\geq 1$.  Then $\Psi^{-1}\circ f\circ\Psi$ has $[v]$ as a characteristic direction and:
\begin{enumerate}
\item $t$ is invariant if $\sigma>t-k$,
\item $r$ is invariant if $\sigma>r-k$, and
\item $s$ is invariant if $\sigma>\max\{s-t,\frac{s-k}{2}\}$.
\end{enumerate}
Therefore, $k,t,r,s$ are invariant under conjugation by $\Psi$ if $\sigma>\max\left\{r-k,s-t,\frac{s-k}{2}\right\}$.  Consequently, if $s>r+t-k$, then $k,t,r,s$ are invariant under conjugation by $\Psi$ if $\sigma> s-t$.\end{propC}

From Proposition~\hyperref[propC]{C}, we see that a map $f$ as in Theorem~\hyperref[thmA]{A} is invariant under conjugation by any biholomorphism that fixes $p$ and $[v]$ (up through degree $s-t+1$).  However, if the biholomorphism moves $[v]$ in its terms of degree less than $s-t+1$, conjugating $f$ by that biholomorphism can change the relative sizes of $t,r,s$.  Going the other direction, we see that we can sometimes conjugate a map that does not satisfy the conditions of Theorem~\hyperref[thmA]{A} to one that does, hence showing the existence of a domain of attraction to $p$ along $[v]$ for a wider class of maps.\\

In the study of the local dynamics of holomorphic self-maps of $\C^n$ that fix a point and are tangent to the identity, generalizing the concepts behind the well-known Leau-Fatou Flower Theorem in $\C$ has been a driving force (see \cite{CG, M}).  Two of the main objects that arise in the Leau-Fatou Flower Theorem are: (1) attracting directions, which correspond to real lines, and (2)  domains of attraction.  The idea of attracting directions in $\C$ is generalized to characteristic directions in higher dimensions, which correspond to complex lines.  The idea of invariant attracting petals in $\C$ is generalized to $\C^n$ in two main ways: parabolic curves and invariant domains of attraction.  In this paper, we focus on the latter.  \\

In higher dimensions, there are several results that work towards classifying germs of holomorphic self-maps of $\C^n$ that are tangent to the identity at a fixed point based on the existence (or non-existence) of a domain of attraction along a characteristic direction.  In $\C^2$, characteristic directions for non-dicritical maps are split into three types: Fuchsian, irregular, and apparent.  In the Fuchsian case, Hakim, Vivas, Rong, and the author 
independently showed sufficient conditions for the existence (or non-existence) of a domain of attraction along a Fuchsian characteristic direction (see \cite{H1,H2,L1,Ro1,V2}).  In the irregular case, Vivas and the author (for a unique non-degenerate characteristic direction) independently showed that there always is a domain of attraction along an irregular characteristic direction (see \cite{L1,L2,V2}).  In the apparent case, Rong and Vivas independently showed sufficient conditions for the existence (or non-existence) of a domain of attraction along an apparent characteristic direction (see \cite{Ro2,V2}).  Theorem~\hyperref[thmA]{A} extends what is known about the existence of a domain of attraction when the direction is Fuchsian, apparent, or dicritical.  Refer to \S\ref{summary} for a brief summary of results in $\C^2$ on the existence of a domain of attraction whose points converge along a given characteristic directions.  In higher dimensions ($\C^n,n\geq 3$), there are also results on the existence (and non-existence) of a domain of attraction along a characteristic direction, however this has been studied less extensively and is not discussed in this paper (see \cite{AR,H1,H2,SV}).   \\

This paper is organized in the following way.  In \S\ref{prelim}, we introduce the main definitions that will be used throughout this paper.  In \S\ref{BiholoSec} we prove some properties about biholomorphisms, their inverses, and how they affect maps tangent to the identity under conjugation.  The main results we prove in this section are Theorem~\hyperref[thmB]{B} and Proposition~\hyperref[propC]{C}.  We will use these results to see when conditions (1)-(3) of Theorem~\hyperref[thmA]{A} are invariant.  We will also see that by choosing an appropriate biholomorphism we may be able to conjugate a map that does not satisfy conditions (1)-(3) of Theorem~\hyperref[thmA]{A} to one that does.  In \S\ref{setup} we see how conditions (1)-(3) of Theorem~\hyperref[thmA]{A} affect what form $f$ and, consequently, $f^{-1}$ must have and we introduce definitions specific to this set-up.  In \S\ref{proof} we use tools from the previous sections to help prove Theorem~\hyperref[thmA]{A}.  For maps that satisfy the conditions of Theorem~\hyperref[thmA]{A}, we show the rate at which points in the domain of attraction converge (see Proposition~\ref{size}).  Lastly, in \S\ref{summary}, we summarize known results in $\C^2$ on the existence of a domain of attraction whose points converge along a characteristic direction based on properties of the map and characteristic direction.  We include how Theorem~\hyperref[thmA]{A} contributes to this classification.

\subsection*{Acknowledgements}
The author would like to thank Laura DeMarco for useful conversations about this paper.  The author would also like to thank Mattias Jonsson, Liz Vivas, and Marco Abate for comments on an earlier draft. 

\section{Preliminaries} 
\label{prelim}
\begin{Def}A germ of a holomorphic self-map of $\C^n$ that fixes $p\in\C^n$ is \textit{tangent to the identity} at $p$ if the Jacobian $df_p$ is the identity matrix.  Let $\TId(\C^n,p)$ be the set of all such maps.  \end{Def}

Assume, without loss of generality, that $p$ is the origin, $O$, as we can move $p$ to $O$ via conjugation by a linear map.  Near the origin, every $f\in\TId(\C^n,\0)$ such that $f\not\equiv\Id$ can be written as:\begin{equation}\label{standardf}
f(z)=z+P_{k+1}(z)+P_{k+2}(z)+\cdots,\end{equation}
where $P_j:=(p_j,q_j)$ is a homogeneous polynomial of degree $j$ and $P_{k+1}\not\equiv 0$ for some integer $k>0$.
\begin{Def} The \textit{order} of $f\in\TId(\C^n,\0)$ is $k+1$.\end{Def}

\begin{Def}Let $Q:\C^n\to\C^n$ be a homogeneous polynomial and suppose that $Q(v)=\lambda v$ for $v\in\C^n\setminus\{O\}$ and $\lambda\in\C$.  The projection of $v$ to $[v]\in\p^{n-1}(\C)$ is a \textit{characteristic direction}; $[v]$ is \textit{degenerate} if $\lambda=0$ and \textit{non-degenerate} if $\lambda\neq 0$.  \end{Def}

\begin{Def}Let $f\in\TId(\C^n,\0)$ have order $k+1$ with homogeneous expansion as in \eqref{standardf}.  Then $[v]\in\p^{n-1}(\C)$ is a \textit{characteristic direction of degree $s$} if it is a characteristic direction of $P_{k+1},\ldots,P_{s}$, where $s\geq k+1$.  In addition, $[v]$ is \textit{non-degenerate in degree $r+1$} if it is degenerate for $P_{k+1},\ldots,P_r$ and non-degenerate for $P_{r+1}$, where $r+1>k+1$.  In accordance with previous terminology, a \textit{characteristic direction} is a characteristic direction of degree $k+1$.\end{Def}

This definition extends the standard definition of characteristic direction, which was previously reserved for degree $k+1$.  For $f\in\TId(\C^n,\0)$ of order $k+1$ with characteristic direction $[v]$, previous results on the existence of a domain of attraction to $O$ along $[v]$ have relied heavily upon properties of $P_{k+1}$.  In Theorem~\hyperref[thmA]{A}, we rely more on higher degree terms of $f$ (i.e., that $[v]$ is characteristic direction of degree $s$) to show that, given some additional assumptions, a domain of attraction to $O$ along $[v]$ exists. \\

Suppose $f\in\TId(\C^2,\0)$ is of order $k+1$.  A characteristic direction of degree $s\geq k+1$ is sent to a characteristic direction of degree $s$ under conjugation by any biholomorphism $\Psi$ that fixes $O$ and, near $O$, is of the form $\Psi(z,w)=L(z,w)+\BO\left((z,w)^{s-k+1}\right)$, where  $L$ is linear.  In addition, a characteristic direction $[v]$ of degree $s$ is preserved under conjugation by any biholomorphism that fixes $O$ and $[v]$ (see Lemma \ref{invdir}).  \\

For the rest of this paper, we restrict to dimension two and maps $f\in\TId(\C^2,\0)$ of order $k+1$ with characteristic direction $[v]$ at $O$.  Without loss of generality, we also assume that $[v]=[1:0]$. 

\begin{Def}Let $r_j(z,w):=z q_j(z,w)-w p_j(z,w)$, where $P_j:=(p_j,q_j)$.  Let $m_j,l_j$, and $n_j$ be the orders of vanishing of $p_j(1,u), q_j(1,u)$, and  $r_j(1,u)$ at $u=0$, respectively.  \end{Def}

\begin{rem}\label{mln}
Note that $[a:b]$ is a characteristic direction of $P_j$ if and only if $r_j(a,b)=0$.  When $j=k+1$, we simplify the notation so that $m:=m_{k+1},l:=l_{k+1}$, and $n:=n_{k+1}.$  
\end{rem}

We can now explain the distinction between characteristic directions introduced by Abate and Tovena in \cite{AT} and discussed further in \cite{L1,V2}.  In this paper, we will mostly use these distinctions to explain previous results, summarized in \S\ref{summary}, since Theorem~\hyperref[thmA]{A} applies to all of these types of characteristic directions.  The relative size of the orders of vanishing ($m,l,n$) play an important role in the existence of a domain of attraction along the direction $[1:0]$; we use these orders to distinguish between different types of characteristic directions.  We use the relative orders instead of the precise orders $m,l,n$ since a linear change of coordinates that fixes $O$ and $[1:0]$ may change $m, l, n$, but it does not change the relative size of $1+m$ and $n$.

\begin{Def} \label{chardirtypes}
The origin is \textit{dicritical} when $n=\infty$, in which case all directions are characteristic directions.  When the origin is non-dicritical, the characteristic direction $[1:0]$ is:
\begin{enumerate}
\item \textit{Fuchsian} if $1+m=n$
\item \textit{irregular} if $1+m<n$, or
\item \textit{apparent} if $1+m>n>0$. 
\end{enumerate}
\end{Def}

\begin{rem}The direction $[1:0]$ is a characteristic direction of $f$ if and only if $n>0$.  Furthermore, $[v]$ is a non-degenerate characteristic direction of $f$ if and only if $m=0<n$.  The definitions of Fuchsian, irregular, and apparent are invariant under any holomorphic change of coordinates, $\Phi=L+\BO\left((z,w)^2\right)$, whose linear term, $L$, fixes $O$ and $[1:0]$.  
\end{rem}

We will use the following definition in the summary of results in $\C^2$ (see \S\ref{summary}).
\begin{Def}\label{director}
Suppose $[v]=[1:u_0]$ is a characteristic direction of $f$.  \textit{Abate's index} of $[v]$ is $\displaystyle\underset{u=u_0}{\mathrm{Res}}\left(\frac{p_{k+1}(1,u)}{r_{k+1}(1,u)}\right)$.  If $[v]$ is non-degenerate, the \textit{director} of $[v]$ is $\displaystyle\frac{1}{k}\left(\frac{r_{k+1}'(1,u)}{p_{k+1}(1,u)}\right)\bigg|_{u=u_0}$.\end{Def}

We introduce the following definition, which is used in Theorem~\hyperref[thmA]{A}.

\begin{Def}\label{degone}
$[1:0]$ is of \textit{order one in degree $t+1$} if $l_{t+1}=1<l_j$ for all $k+1\leq j\leq t$.   \end{Def}

Equivalently, $[1:0]$ is of \textit{order one in degree $t+1$} $\Leftrightarrow$ $w | q_{t+1}(z,w), w^2\not|q_{t+1}(z,w)$ and $w^2 | q_j(z,w)$ for all $k+1\leq j\leq t$.  Note that $[1:0]$ is a characteristic direction of $P_j~\Leftrightarrow ~l_j\geq 1~\Leftrightarrow~n_j\geq1$.  If $[1:0]$ is of order one in degree $t+1$, then $[1:0]$ is a characteristic direction of degree at least $t+1$.

\section{Biholomorphisms}\label{BiholoSec}

In this section we discuss properties of biholomorphisms that fix a point.  We focus on biholomorphisms of the form $\Psi\in\TId(\C^n,\0)$ since, by composing with a linear isomorphism, we can always move the fixed point to the origin and make the linear term the identity.  We show how some properties of a biholomorphism translate to its inverse as well as how properties of a biholomorphism affect a map $f\in\TId(\C^n,\0)$ under conjugation.  \\

The following lemma is an extension of \cite[Lemma 1.1]{A1} from order 2 to order $\tau\geq 2$.
\begin{lem}\label{inverse} Let $\Psi\in\TId(\C^n,\0)$ be of order $\tau$.  Near the origin, 
$$\Psi(z)=z+\sum_{j\geq\tau}A_j(z)\qquad\mbox{ and }\qquad\Psi^{-1}(z)=z+\sum_{j\geq\tau}B_j(z),$$
where $A_j=(A_j^{1},\ldots,A_j^{n})$ and $B_j=(B_j^{1},\ldots ,B_j^{n})$ are homogeneous of degree $j$, and $A_\tau\not\equiv O$.  Let $z=(z^1,\ldots,z^n)$.  Then:
$$B_j=-A_j,~\forall j\leq 2(\tau-1)\quad\mbox{ and }\quad
B_{2\tau-1}=-A_{2\tau-1}+\sum_{l=1}^n A_\tau^l \frac{\partial A_\tau}{\partial z^l}.$$
\end{lem}

\begin{proof}Since $\Psi\in\TId(\C^n,\0)$ is of order $\tau$, clearly $\Psi^{-1}\in\TId(\C^n,\0)$ is also of order $\tau$.  We use Taylor series to approximate terms in $\Psi^{-1}$; in particular, we use the formula:
$$A_j(z+\zeta)=A_j(z)+\sum_{l=1}^n \zeta^l \frac{\partial A_j}{\partial z^l}(z)+\BO(||\zeta||^2),$$
where $\zeta=(\zeta^1,\ldots,\zeta^n)$.  Then: 
\begin{align*}
z&= \Psi\circ\Psi^{-1}(z)
 =z+\sum_{j\geq \tau}B_j(z)+\sum_{j\geq \tau}A_j\Big(z+\sum_{t\geq \tau}B_t(z)\Big) \\
& =  z+\sum_{j\geq\tau}B_j(z)+\sum_{j\geq \tau}
   \left(A_j(z)+\sum_{l=1}^n \Big(\sum_{t\geq \tau}B_t^l(z) \Big) \frac{\partial A_j}{\partial z^l}(z)+\BO\Big(\Big|\Big|\sum_{t\geq \tau}B_t(z)\Big|\Big|^2\Big)\right) \\
   &=z+\sum_{j=\tau}^{2\tau-1}\left(B_j(z)+A_j(z)\right)
    +\sum_{l=1}^n\Big(B_\tau^l(z)\frac{\partial A_\tau}{\partial z^l}(z) \Big) 
     +\BO(||z||^{2\tau})
\end{align*}
Therefore, $B_j=-A_j$ for $j<2\tau-1$ and $\displaystyle B_{2\tau-1}=-A_{2\tau-1}+\sum_{l=1}^n A_\tau^l\frac{\partial A_\tau}{\partial z^l}$.  
\end{proof}

\begin{thmB}
Let $\Psi$ be a biholomorphism of $\C^n$ that fixes the origin, $\0$.  Near the origin:
$$\Psi(z,w)=\sum_{j\geq1} \Psi_j(z,w),$$ 
where $(z,w)\in\C\times\C^{n-1}$ and $\Psi_j$ is homogeneous of degree $j$.  Let $[1:\0]:=[1:0:\cdots:0]\in\p^{n-1}(\C)$.  If $\Psi$ fixes $[v]\in\p^{n-1}(\C)$ up through degree $\sigma$, then:
\begin{enumerate}
\item each $\Psi_j$ fixes $[v]$ for all $j\leq\sigma$;
\item $\Psi^{-1}$ fixes $[v]$ up through degree $\sigma$; and
\item given an invertible linear map $L$ that fixes $\0$ and sends $[1:\0]$ to $[v]$, there is a holomorphism $\phi:\C\to\C^{n-1}$ and a biholomorphism $\chi:\C^{n}\to\C^{n}$ that fixes $\0$ and $[1:\0]$ such that:
$$L^{-1}\circ\Psi\circ L=\Phi\circ \chi,$$
where $\Phi=\Id+(0,z^{\sigma+1}\phi(z))$ is a biholomorphism.
\end{enumerate} 
 \end{thmB}

\begin{proof}
Without loss of generality, we move $[v]$ to $[1:\0]$ via a linear change of coordinates, $L$.  We say that $\Psi$ fixes $[1:\0]$ up through degree $\sigma$ if we can write
$\Psi(z,w)=\widetilde{\Psi}(z,w)+\BO\left( (z,w)^{\sigma+1}\right)$, where $\widetilde{\Psi}=\sum_{j=1}^{\sigma}\Psi_j$ fixes $\0$ and $[1:\0]$.  Then, for all $z\in\C$ near $0$:
$$\widetilde{\Psi}(z,\0)=\sum_{j=1}^{\sigma}\Psi_j(z,\0)=\sum_{j=1}^{\sigma}(a_j,b_j) z^j=\lambda_z (1,\0),$$
where $(a_j,b_j):=\Psi_j(1,\0)\in\C\times\C^{n-1}$ and $\lambda_z\in\C$ is a constant.  Then $\sum_{j=1}^\sigma b_j z^j=\0$ for all $z\in\C$ near $0$, which implies that $b_j=\0$ for all $j\leq\sigma$.  Hence, for all $j\leq\sigma$, each $\Psi_j$ fixes $[1:\0]$. \\

Now we show that $\Psi^{-1}$ fixes $[1:\0]$ up through degree $\sigma.$  We can express $\Psi$ as:
\begin{equation}\label{biholo}
\Psi(z,w)=\big(A_1(z,w),A_2(z,w)+z^{\sigma+1} A_3(z)\big)\in\C\times\C^{n-1}\end{equation}
for $(z,w)\in\C\times\C^{n-1}$, where $A_1,A_2,A_3$ are holomorphic and $A_2(z,\0)=\0$.  Let:
$$\Psi^{-1}(z,w)=(B_1(z,w),B_2(z,w)+\hat{B}(z))\in\C\times\C^{n-1},$$ 
where $B_1,B_2,\hat{B}$ are holomorphisms such that $B_2(z,\0)=\0$ and $\hat{B}(0)=\0$.  We now prove \textit{(2)} by showing that $\hat{B}(z)=\BO(z^{\sigma+1})$.  Since $\Psi$ is a biholomorphism that fixes $\0$ and $[1:\0]$, $A_1(z,\0)$ must have a non-zero linear term and we can write:
\begin{equation}\label{A1}
A_1(z,w):=z\alpha_1(z)+\alpha_2(z,w):=\eta(z)+\alpha_2(z,w),\end{equation}
where $\alpha_2(z,\0)=0$, $\alpha_1(0)\neq0$, and $\eta$ has an inverse around $0$.  We use that $\alpha_1(0)\neq0$ and that:
$$(z,\0)=\Psi^{-1}\circ\Psi(z,\0)=
\left(B_1\circ\Psi(z,\0),  
B_2\left(A_1(z,\0),z^{\sigma+1}A_3(z)\right)+\hat{B}\left( A_1(z,\0)\right)\right)$$
to see that:
$$\0=B_2\left(A_1(z,0),z^{\sigma+1}A_3(z)\right)+\hat{B}\left( A_1(z,\0)\right)=
\BO(z^{\sigma+1})+\hat{B}(z\alpha_1(z)).$$
Hence, $\hat{B}(z)=\BO(z^{\sigma+1})$.  Therefore $\Psi^{-1}$ fixes $[1:0]$ through degree $\sigma$ and we can express $\Psi^{-1}$ as:
\begin{equation}\label{biholo2}
\Psi^{-1}(z,w)=\big(B_1(z,w),B_2(z,w)+z^{\sigma+1}B_3(z)\big),\end{equation}
where $B_1,B_2,B_3$ are holomorphisms that depend on $\Psi$ and $B_2(z,\0)=\0$. \\

Finally, we show that $\Psi$ can be rewritten as the composition of two functions:  
$$\chi(z,w)=(A_1(z,w),\hat{A}_2(z,w))\quad\mbox{and}\quad
\Phi(z,w)=\Id+(0,z^{\sigma+1}\phi(z)),$$
where $\hat{A}_2(z,\0)=\0$ since $\chi$ fixes $[1:\0]$.  We will define $\hat{A}_2$ and $\phi$ so that $\Phi\circ\chi-\Psi=\0:$  
\begin{equation}\label{lemmaeq}
\Phi\circ\chi(z,w)-\Psi(z,w) =\left(0,\hat{A}_2(z,w)+A_1(z,w)^{\sigma+1}\phi(A_1(z,w))-A_2(z,w)-z^{\sigma+1}A_3(z)\right).\end{equation}
Using \eqref{A1} and $w=\0$ in \eqref{lemmaeq} we see that: 
\begin{equation}\label{zero}
\Phi\circ\chi(z,\0)-\psi(z,\0)=\left(0,
\eta(z)^{\sigma+1}\phi(\eta(z))-z^{\sigma+1}A_3(z)\right).
\end{equation}
We define $\phi(z)$ so that \eqref{zero} equals $\0$.  In particular,
\begin{equation}\label{phi}
\phi(z):=\left(\frac{\eta^{-1}(z)}{z}\right)^{\sigma+1}A_3(\eta^{-1}(z)).
\end{equation}
Since $\eta(z)=z\alpha_1(z)$ is holomorphic and invertible around $0$ with $\alpha_1(0)\neq 0$ and $\eta(0)=0$, we know that $\eta^{-1}(z)=z\beta_1(z)$ for some holomorphic $\beta_1$.  Given that $\beta_1(z)=\frac{\eta^{-1}(z)}{z}$ and $A_3(\eta^{-1}(z))$ are holomorphic near $0$, $\phi(z)$ must also be holomorphic near $z=0$.  \\

After setting \eqref{lemmaeq} equal to $\0$, plugging in $\phi$ from \eqref{phi}, and solving for $\hat{A}_2$, we get:
$$\hat{A}_2(z,w):=\underbrace{A_2(z,w)}_{\footnotesize =\0\mbox{ if }w=\0}+\underbrace{[z^{\sigma+1} A_3(z)-A_1(z,w)^{\sigma+1} \phi(A_1(z,w))]}_{\footnotesize =\0\mbox{ if }w=\0\mbox{ by definition of }\phi}.$$
We see that $\hat{A}_2$ is holomorphic and $\hat{A}_2(z,\0)=\0$.  Therefore, $\Psi$ can be rewritten as the composition of $\chi$ and $\Phi$ (i.e., $\Psi=\Phi\circ\chi$).  In addition, $\chi$ must be a biholomorphism near $\0$ since $\Psi$ and $\Phi$ are biholomorphisms near $\0$, with $\Phi^{-1}(z,w)=\Id-(0,z^{\sigma+1}\phi(z))$, and so $\chi=\Phi^{-1}\circ\Psi.$
 \end{proof}

In dimension $n=2$, \eqref{biholo} and \eqref{biholo2} can be rewritten as:
\begin{align}\label{biholo3}
\Psi(z,w) &=\big(A_1(z,w),wA_2(z,w)+z^{\sigma+1}A_3(z)\big)\mbox{ and}\\
\Psi^{-1}(z,w)&=\big(B_1(z,w),wB_2(z,w)+z^{\sigma+1}B_3(z)\big),\notag\end{align}
where we replaced $A_2,B_2$, which satisfied $A_2(z,0)=B_2(z,0)=0$,  by $wA_2,wB_2$.

\begin{lem}\label{invdir}
Let $f\in\TId(\C^2,\0)$ be of order $k+1$ and have $[v]$ as a characteristic direction of degree $s$.  Suppose that $\Psi$ is a biholomorphism of $\C^2$ that fixes the origin and, near the origin, fixes $[v]$.  Then $\Psi^{-1}\circ f\circ\Psi$ has $[v]$ as a characteristic direction of degree $s$.
\end{lem}
\begin{proof}
Without loss of generality, we can assume $[v]=[1:0]$ since we can move it there via a linear conjugation.  We can write $\Psi,\Psi^{-1}$ as in \eqref{biholo3} with $A_3=B_3=0$ and $f$ as: 
$$f(z,w)=(f_1(z,w),wf_2(z,w)+z^{s+1}S(z)),$$
where $f_1, f_2,S(z)$ are power series.  Then:
\begin{align*}
\Psi^{-1}\circ f\circ\Psi(z,w)
&=\Big(B_1( f(\Psi(z,w))), wA_2(z,w)f_2(\Psi(z,w))B_2(f(\Psi(z,w)))+\BO((z,w)^{s+1}\Big), 
\end{align*}
which fixes the direction $[1:0]$ up through degree $s$.
\end{proof}


Now we want to write an expression for a map $f$ that satisfies conditions (1)-(3) of Theorem~\hyperref[thmA]{A}.  For $f\in\TId(\C^2,\0)$ of order $k+1$, we already saw that near the origin:
$$f(z,w)=\Id+P_{k+1}(z,w)+P_{k+2}(z,w)+\cdots.$$
If we also assume that $f$ satisfies conditions (1)-(3) of Theorem~\hyperref[thmA]{A}, we can express $f$ as:
$$f(z,w)=\Id+\underbrace{w\left(\BO((z,w)^k),w\BO((z,w)^{k-1})\right)}_{\footnotesize [1:0]\mbox{ degen. char. dir. }P_{k+1},\ldots,P_r}+\big(\underbrace{z^{r+1}\BO(1)}_{\footnotesize \mbox{non-degen.}},\underbrace{wz^t\BO(1)}_{\footnotesize \mbox{order one}}+\underbrace{z^{s+1}\BO(1)}_{\footnotesize \mbox{not char. dir.}} \big).$$
To make $f$ easier to work with, we re-express $f$ as:
\begin{align}\label{feq0}
f(z,w)&=\left(z(1+z^rR(z))+wU(z,w),w(1+z^tT(z)+wV(z,w))+z^{s+1}S(z)\right),\mbox{ or} \\
\label{feq0b} &=(f_1(z,w),wf_2(z,w)+z^{s+1}S(z)), 
\end{align}
where: $R,S,T,U,V,f_1,f_2$ are holomorphic in a neighborhood of the origin; $R(0),T(0)$ are non-zero; and the power series expansions of $U,V$ near $\0$ have terms of lowest degree $k,k-1$, respectively.  

\begin{propC}
Let $f$ be a germ of a holomorphic self-map of $\C^2$ that is tangent to the identity at a fixed point $p$, is of order $k+1$, has characteristic direction $[v]$, and satisfies conditions (1)-(3) of Theorem~\hyperref[thmA]{A}.  Suppose that $\Psi$ is a biholomorphism of $\C^2$ that fixes $p$ and, near $p$, fixes $[v]$ up through degree $\sigma\geq 1$.  Then $\Psi^{-1}\circ f\circ\Psi$ has $[v]$ as a characteristic direction and:
\begin{enumerate}
\item $t$ is invariant if $\sigma>t-k$,
\item $r$ is invariant if $\sigma>r-k$, and
\item $s$ is invariant if $\sigma>\max\{s-t,\frac{s-k}{2}\}$.
\end{enumerate}
Therefore, $k,t,r,s$ are invariant under conjugation by $\Psi$ if $\sigma>\max\left\{r-k,s-t,\frac{s-k}{2}\right\}$.  Consequently, if $s>r+t-k$, then $k,t,r,s$ are invariant under conjugation by $\Psi$ if $\sigma> s-t$.\end{propC}

\begin{rem}\label{tauforinv}
If we make the additional assumption from Theorem~\hyperref[thmA]{A} that $s>r+t-k$, then the condition on $\sigma$ in Proposition~\hyperref[propC]{C} simplifies to $\sigma>s-t$.  Also, from Proposition~\hyperref[propC]{C}, we can see that it is (sometimes) possible to conjugate a map that does not satisfy the conditions of Theorem~\hyperref[thmA]{A} to one that does satisfy these conditions by using a biholomorphism that does not fix $[v]$. \end{rem}

\begin{proof}  \label{proofpropC}
Without loss of generality, we assume $[v]=[1:0]$ since we can move it there via a linear conjugation.  We will use the decomposition of $\Psi=\Phi\circ\chi$ from Theorem~\hyperref[thmB]{B} to help show that $k,t,r,s$ are invariant under conjugation by $\Psi$ for $\sigma$ sufficiently large.  We already know that $k+1$, the order of $f$, is invariant under conjugation by any biholomorphism that fixes the origin. \\

First of all, we will show that $t,r,s$ are invariant under conjugation by $\chi$, which fixes $[1:0]$:
$$\chi(z,w)=(A_1(z,w),w\hat{A}_2(z,w))\mbox{ and }\chi^{-1}(z,w)=(B_1(z,w),w\hat{B}_2(z,w)).$$  
Secondly, we will show how $t,r,s$ can vary under conjugation by $\Phi$:
$$\Phi(z,w)=\Id+(0,z^{\sigma+1}\phi(z))\mbox{ and }\Phi^{-1}(z,w)=\Id-(0,z^{\sigma+1}\phi(z)).$$
 Since $\Psi=\Phi\circ\chi$, we will conclude that $\sigma$ determines which of $t,r,$ and $s$ are invariant. \\

We first conjugate $f$ by $\chi$, using \eqref{feq0b} and that $k\leq t\leq r<s$, to get $\chi^{-1}(f(\chi(z,w)))$ equals:
\begin{align}\label{conjf}
 \left(B_1(f(\chi(z,w))),
 \left(w\hat{A}_2(z,w) f_2(\chi(z,w))+A_1^{s+1}(z,w)S(A_1(z,w))\right)\hat{B}_2(f(\chi(z,w))) \right) 
\end{align}
To see that $t$ is invariant, we focus on $f_2(\chi(z,0))$ since $\hat{B}_2,\hat{A}_2$ have non-zero constant terms and $A_1^{s+1}S(A_1(z,w)))=\BO((z,w)^{s+1})$ with $s>t$:
$$f_2(\chi(z,0))=1+(A_1(z,0))^t T(A_1(z,0)).$$
Since $A_1$ has a linear term in $z$ and $T(0)\neq0$, we see that the $f_2\circ\chi$ has a non-zero term $z^t$ and no lower degree terms in just $z$.  Therefore $t$ is invariant under $\chi$. \\

To show that $r$ is invariant, we look at the first coordinate of \eqref{conjf} and let $w=0$:
\begin{align*}
B_1\circ f\circ\chi(z,0) &=
B_1\left(A_1(1+A_1^{r}R\circ A_1),\BO(A_1^{s+1})\right) \\
& =B_1\left(A_1(1+A_1^{r}R(0)),0\right)+\BO(z^{r+2}).
\end{align*}
Using that $B_1(z,0)=\sum b_j z^j$, $b_1\neq0$, and $B_1(A_1(z,0),0)=z$ we see that:
$$B_1\left(A_1(1+A_1^{r}R(0)),0\right)
=\sum_{j\geq 1} b_j A_1^j(1+A_1^rR(0))^j
=z+\alpha z^{r+1}+\BO(z^{r+2}),$$
where $\alpha\neq0$ since $R(0)\neq0$.  Therefore $r$ is invariant under $\chi$. \\

To show that $s$ is invariant, we look at the second coordinate of \eqref{conjf} and let $w=0$:
$$A_1^{s+1}(z,0)S(A_1(z,0)) \B(f(\chi(z,0)))=\beta z^{s+1}+\BO(z^{s+2}),$$
for some $\beta\neq0$.  We used that $S(0)\B(0,0)\neq0$ and $A_1$ has a non-zero linear term in $z$.  Therefore $s$ is invariant under $\chi$.  Consequently, $k,t,r,s$ are invariant under conjugation by any biholomorphism, such as $\chi$, that preserves $\{w=0\}$. \\

Now we show how conjugation by $\Phi$ can affect $t,r,s$ depending on the size of $\sigma$.  If $\sigma$ is sufficiently large, then $t,r,s$ are invariant under conjugation by $\Phi$.  Note that: 
$$\Phi(z,w)=(z,w+z^{\sigma+1}\phi(z))\qquad\mbox{ and }\qquad\Phi^{-1}(z,w)=(z,w-z^{\sigma+1}\phi(z)).$$  
We will be using our assumption that $f$ satisfies conditions (1)-(3) of Theorem~\hyperref[thmA]{A}; in particular, that $k\leq t\leq r<s$ and $k<r$.  \\

Using the expression for $f$ in \eqref{feq0}, we see that $\Phi^{-1}\circ f\circ\Phi(z,w):=(z_1,w_1)$ is:
\begin{align}\label{conjfinv}
z_1&= z(1+z^r R(z))+(w+z^{\sigma+1}\phi(z))U(z,w+z^{\sigma+1}\phi(z)) \\ 
  &= z\left(1+z^r R(z)+z^{\sigma}\phi(z)U(z,z^{\sigma+1}\phi(z))\right)+w\BO((z,w)^k) \notag \\ 
  &=z\left(1+z^rR(z)+\BO(z^{\sigma+k})\right)+w\BO((z,w)^k) \notag
\end{align}
\begin{align}\label{conjfinv2}
w_1 &=(w+z^{\sigma+1}\phi(z))(1+z^tT(z)+(w+z^{\sigma+1}\phi(z))V(\Phi(z,w)))+z^{s+1}S(z)-z_1^{\sigma+1}\phi(z_1)   \\ 
&=w\left(1+z^tT(z)+ \BO(z^{\sigma+k})
+w\BO((z,w)^{k-1})\right)+z^{\sigma+1}\phi(z)+\BO\left(z^{\sigma+1+t},z^{2\sigma+1+k}\right) \notag \\ 
&~~\mbox{         }~+z^{s+1}S(z)-\Big(z^{\sigma+1}\big(1+\BO(z^{r},z^{\sigma+k})\big)+w\BO((z,w)^{\sigma+k})\Big)\phi(z_1) \notag \\
&=w\left(1+z^tT(z)+ \BO(z^{\sigma+k})
+w\BO((z,w)^{k-1})\right)+z^{s+1}S(z)+\BO\left(z^{\sigma+1+t},z^{2\sigma+1+k}\right). \notag
 \end{align} 

From \eqref{conjfinv} and \eqref{conjfinv2}, we see that $t,r,s$ are invariant given the following conditions on $\sigma:$
\begin{enumerate}
\item $t$ is invariant if $\sigma>t-k$,
\item $r$ is invariant if $\sigma>r-k$, and
\item $s$ is invariant if $\sigma>\max\{s-t,\frac{s-k}{2}\}$.
\end{enumerate}
Since $t,r,s$ are invariant under $\chi$, we now see that $t,r,s$ are invariant under $\Psi=\Phi\circ\chi$ if:
$$\sigma>\max\left\{r-k,s-t,\frac{s-k}{2}\right\}.$$
 \end{proof}
 
Notice that, depending on the relative sizes of $k,t,r,s$, it is possible that $\sigma$ is large enough for $s$ to be invariant without $t,r$ being invariant or that $t,r$ are invariant without $s$ being invariant.  \\

\section{Set-Up}\label{setup}
Now we use some of the assumptions in Theorem~\hyperref[thmA]{A} to rewrite $f$.  In particular, assume $[1:0]$ is:
\begin{enumerate}
\item a characteristic direction of degree $s$, where $s\leq\infty$;   
\item non-degenerate of degree $r+1$, where $k<r<s$; and
\item of order one in degree $t+1$, where $k\leq t\leq r$. \\
\end{enumerate} 
%

Given these assumptions, we saw in \eqref{feq0} that we can express $f$ as follows:
\begin{equation}
\label{feq1}
f(z,w)=\left( z(1+z^rR(z))+wU(z,w),~w(1+z^tT(z)+wV(z,w))+z^{s+1}S(z)\right),
\end{equation}
where $R,S,T,U,V$ are holomorphic near the origin.  Furthermore, $R,S,T$ satisfy these properties:
\begin{align}
\label{RST}
\bullet~ & c:=S(0)\neq0,\mbox{ so that }S(z)=c+\BO(z)\mbox{ near }0,\mbox{ or }s=\infty\mbox{ and }S\equiv0;\\
\bullet~ & a:=R(0)\neq0,\mbox{ so that }R(z)=a+\BO(z) \mbox{ near }0; \mbox{ and} \notag \\
\bullet~ & b:=T(0)\neq0,\mbox{ so that }T(z)=b+\BO(z) \mbox{ near }0. \notag
\end{align}

In addition, we can bound $U,V$ using the orders of vanishing $m,l$ given in Remark \ref{mln} so that:
\begin{align}
\label{UV}
U(z,w) &=w^{m-1}\BO((z,w)^{k+1-m})+\BO((z,w)^{k+1})\mbox{ and}\\ V(z,w)&=w^{\max\{l-2,0\}}\BO\left((z,w)^{k-1-\max\{l-2,0\}}\right)+\BO((z,w)^{k}),\notag
\end{align}
where $m,l\geq 1$ since $[1:0]$ is a degenerate characteristic direction.  If we simplify these bounds to ignore the specific values of $m,l$ we see that $U(z,w)=\BO((z,w)^k)$ and $V(z,w)=\BO((z,w)^{k-1})$. \\

Any invertible linear map fixing the origin and $[1:0]$ must be of the form $L(z,w)=(a_1 z+b_1 w, a_2 w)$, where $a_1a_2\neq 0$. 
We conjugate $f$ in \eqref{feq1} by $L$ (with $b_1=0$) and rename $R,S,T,U,V$ so that they satisfy the same conditions as in \eqref{RST} and \eqref{UV}.  Then $f$ (which is actually $L^{-1}\circ f\circ L$) is: 
\begin{equation}\label{feq2}
\big(z(1+a_1^{r}z^{r}R(z))+wU(z,w),~w\left(1+a_1^tz^tT(z)+wV(z,w)\right)+a_1^{s+1}a_2^{-1}z^{s+1}S(z)\big).
\end{equation}
Let $\beta:=-ba_1^t$ and choose $a_1$ so that $aa_1^{r}=-r^{-1}$.  When possible, choose the $\frac{1}{r}$-root so the director has positive real part, where
we extend the definition of the \textit{director} of $[v]$ (given in \ref{director}) to:
\begin{equation}\label{newdirector}
\Delta:=
\begin{cases}\beta=-b(-ar)^{-\frac{t}{r}}&\mbox{ if }t<r \\ 
\beta-\frac{r-k+1}{r}=\frac{1}{r}\left(\frac{b}{a}-(r-k+1)\right)&\mbox{ if }t=r,\end{cases}.
\end{equation}
If $\frac{r}{gcd(r,t)}>2$ (i.e., $r\neq t, 2t$), we can always choose the $\frac{1}{r}$-root so that $\re(\Delta)>0$.  Otherwise, we might have $\re(\Delta)>0$, but it is not guaranteed.  The definition of director is invariant under any biholomorphic change of coordinates that fixes $[v]$ or, more generally, that leaves $t,r$ invariant as we can see from the proof of Proposition~\hyperref[proofpropC]{C}

\begin{rem}The definition of director given in \eqref{newdirector} agrees with the standard definition of director (see \ref{director}), the latter of which only applies to the case when $k=t=r$.  
Our definition of director differs from that given in \cite{Ro2} in two ways.  First of all, the definition in \cite{Ro2} only applies to $k=t<r$, whereas our definition applies to $k\leq t\leq r$ and $k\neq r$.  Secondly, our definition is multiplied by $r^{-\frac{t}{r}}$, which does not affect whether $\re(\Delta)>0$.  As we only care about the sign of $\re(\Delta)$, all three definitions effectively agree.\end{rem}


We extend the definition of transversally attracting used in Rong's paper \cite{Ro2} to this setting.  

\begin{Def}\label{transv}
We say that $f$ is \textit{transversally attracting} in $[v]$ if $\re(\Delta)>0.$\end{Def}

Given $f$ and $[v]$ that satisfy conditions (1)-(3) of Theorem~\hyperref[thmA]{A}, then $[v]$ is transversally attracting: always when $r\not\in\{t,2t\}$, almost always when $r=2t$, and sometimes when $r=t$.  \\

In the next lemma we see that in this situation $f^{-1}$ shares many of the same properties as $f$.

\begin{lem}\label{f-1}
Let $f$ be a germ of a holomorphic self-map of $\C^2$ that is tangent to the identity at a fixed point $p$, is of order $k+1$, and satisfies conditions (1)-(3) of Theorem~\hyperref[thmA]{A} with characteristic direction $[v]$.  Then $f^{-1}$ satisfies conditions (1)-(3) of Theorem~\hyperref[thmA]{A} with characteristic direction $[v]$ and the same $k,r,s,t$.  If $r=t$, then $[v]$ is transversally attracting for $f$ if and only if it is for $f^{-1}$.  If $r=2t$, then $[v]$ must be transversally attracting for $f$ or $f^{-1}$. 
\end{lem}

\begin{proof}Without loss of generality, we move $p$ to $\0$ and $[v]$ to $[1:0]$ via a linear change of coordinates.  By Theorem~\hyperref[thmB]{B}, since $[1:0]$ is a characteristic direction of degree $s$ for $f$, it must also be for $f^{-1}$.  We can write $f^{-1}$ in the same way as we wrote $f$ in \eqref{feq0} so that:
\begin{align*}
f(z,w)&=\left(z(1+z^rR(z))+wU(z,w),w(1+z^tT(z)+wV(z,w))+z^{s+1}S(z)\right) \\
f^{-1}(z,w)&=\left(z(1+z^{\hat{r}}\hat{R}(z))+w\hat{U}(z,w),w(1+z^{\hat{t}}\hat{T}(z)+w\hat{V}(z,w))+z^{s+1}\hat{S}(z) \right),
\end{align*}
where $\hat{R},\hat{S},\hat{T},\hat{U},\hat{V}$ are power series, $\hat{R}(0)\neq0$ or $\hat{R}\equiv0$, and $\hat{T}(0)\neq 0$ or $\hat{T}\equiv0$.  Let $\pi_j$ be projection onto the $j$th-coordinate.  Then:
\begin{align*}
z&=\pi_1\circ f^{-1}\circ f(z,0)
 =z+z^{r+1}R(0)+z^{\hat{r}+1}\hat{R}(0) +\BO(z^{r+2},z^{\hat{r}+2},z^{s+1})
\end{align*}
Since $R(0)\neq 0$, $s>r$, and the previous equation holds for all $z$ near $0$, we must have that $\hat{r}=r$ and $\hat{R}(0)=-R(0)$.  Now looking at the second coordinate we see that:
\begin{align*}
w&=\pi_2\circ f^{-1}\circ f(z,w)
    =w+wz^tT(0)+wz^{\hat{t}}\hat{T}(0)+w\BO(w,z^{t+1},z^{\hat{t}+1})+\BO(z^{s+1}).  
\end{align*}
Since $T(0)\neq0,$ $s>t$, and the previous equation holds for all $(z,w)$ near $\0$, we must have that $\hat{t}=t$ and $\hat{T}(0)=-T(0)$.  \\

As in \eqref{RST}, let $a=R(0)=-\hat{R}(0)$ and $b=T(0)=-\hat{T}(0)$.  Let $\Delta$ and $\hat{\Delta}$ be the directors for $f$ and $f^{-1}$, respectively.  When $r=t$, $\Delta=\hat{\Delta}$ and so $[v]$ is transversally attracting for $f$ if and only if it is for $f^{-1}$.  When $r=2t$,
$\hat{\Delta}=b(ar)^{-\frac{1}{2}}=-\Delta (-1)^{-\frac{1}{2}}.$  As long as $\Delta,\hat{\Delta}\not\in\R$, we can choose the square root so $\re\Delta,\re\hat{\Delta}>0$, hence $[v]$ is transversally attracting for both $f$ and $f^{-1}$.  If $\Delta\in\R$, then $\hat{\Delta}\in i\R$ and we can choose the root so $\Delta>0$, hence $[v]$ is transversally attracting for $f$, but not $f^{-1}$.  Similarly, if we switch $\Delta$ and $\hat{\Delta}$, then $[v]$ is transversally attracting for $f^{-1}$, but not $f$.
\end{proof}

We now rewrite $f$ from \eqref{feq2} by using our definition of $a_1,\beta$ and $\lambda:=ca_1^{s+1}a_2^{-1}$:
\begin{equation}\label{feq3}
f(z,w)=\left(z\Big(1-\frac{1}{r}z^{r}R(z)\Big)+wU(z,w),~w\big(1-\beta z^tT(z)+wV(z,w)\big)+\lambda z^{s+1}S(z)\right),\end{equation}
where we renamed $R,S,T$ to satisfy the same bounds as in \eqref{RST} and $R(0)=T(0)=S(0)=1$.  If $c\neq0$, we choose $a_2$ so that $|\lambda|\ll 1$; if $s=\infty$, then $c=\lambda=0$.

\section{Attracting domain}\label{proof}
We now prove the main theorem.  Assume that $[v]$ is transversally attracting (so $\re\Delta>0$).  Let
$$V:=V_{\delta,\theta}:=\left\{z\in\C \ | \ 0<|z|<\delta, |\Arg(z)|<\theta\right\}$$
$$D:=D_{\delta,\theta,\mu}:=\left\{(z,w)\in\C^2 \ | \ z\in V_{\delta,\theta}, |w|<|z|^{\mu}\right\}$$
where $0<\theta\ll \frac{\pi}{4r}$, $0<\epsilon<\frac{1}{3}$, $0<\delta\ll 1$, and    
\begin{equation}\label{eta}
\mu:=\max\left\{\frac{r-k+\epsilon}{m}+1,r-k+\epsilon,t-k+1+\epsilon\right\}\leq r-k+1+\epsilon.\end{equation}
 If $t=r$, choose $\epsilon$ so that $0<\epsilon<r\re\Delta$.  We choose $\delta$ small based on $\epsilon$; in particular, so that 
$0<\delta^\epsilon\ll\min\left\{r^{-1},\re\hat{\beta}\right\}$, 
where $\hat{\beta}$ is defined in \eqref{betahat}.  We want to minimize $\mu$ since we will need $s+1\geq\mu+t+2\epsilon$, but we also need $\mu$ to be sufficiently large, as we'll see in \eqref{etabdd}. 

\begin{rem}\label{smubdd}
In condition (1) of Theorem~\hyperref[thmA]{A} we require $s>r$ and we also assume that $s>r+t-k$.  Given the bounds on $\mu$ in \eqref{eta} and that $s\in\N$, we see that $s+1\geq\mu+t+2\epsilon.$\end{rem}



We use the bounds on $U,V$ from \eqref{UV} and on $w$ for $(z,w)\in D$ to get these bounds:
\begin{align*}
\frac{w}{z}U(z,w)& = \BO(w^m z^{k-m},wz^k)=\BO(z^{k+m(\mu-1)},z^{k+\mu})=\BO(z^{r+\epsilon}) \\
wV(z,w) &=\BO(wz^{k-1})=\BO(z^{k+\mu-1})=\BO(z^{t+\epsilon})
\end{align*}
where the bounds on the right are at least as lenient as those on the left.  We chose $\mu$ to get the bounds on $\frac{w}{z}U$ and $wV$ listed on the far right. In particular, we chose $\mu$ as in \eqref{eta} since:
\begin{align}\label{etabdd}
k+m(\mu-1)&\geq r+\epsilon & \Leftrightarrow && \mu\geq& \frac{r-k+\epsilon}{m}+1, \\
k+\mu&\geq r+\epsilon & \Leftrightarrow && \mu\geq &r-k+\epsilon,\mbox{ and} \notag\\
k+\mu-1&\geq t+\epsilon &\Leftrightarrow && \mu\geq& t-k+1+\epsilon.\notag
\end{align}


We want to show that $f(D)\subset D$ and $(z_n,w_n)\to (0,0)$ along $[1:0]$ as $n\to\infty$.  \\ 

First we show that $D$ is $f$-invariant.  Take any $(z,w)\in D$.  Then:
$$0<|z_1|= |z|\left|1-\frac{1}{r}z^{r}R(z)+\frac{w}{z}U(z,w)\right|< 
\delta \left|1-\frac{1}{r}z^{r}+\BO\left(z^{r+\epsilon}\right)\right|<\delta$$
and 
$$|\Arg(z_1)|=\left|\Arg(z)+\Arg\Big(1-\frac{1}{r}z^r+\BO\left(z^{r+\epsilon}\right)\Big)\right|<\theta.$$
So $z_1\in V$ and now we need to show that $|w_1|<|z_1|^{\mu}$.
 
 \begin{align}
\label{wbound}\frac{|w_1|}{|z_1|^{\mu}} &
	=\frac{|w|}{|z|^{\mu}}
		\left|\frac{1-\beta z^t+\BO\left(z^{t+\epsilon}\right) 
		+\lambda z^{s+1}w^{-1}S(z)} 
		{\left(1-\frac{1}{r}z^r+\BO\left(z^{r+\epsilon}\right)\right)^{\mu}} \right| \\ 
\label{wboundb}	&=\frac{|w|}{|z|^{\mu}}
		\left|
		1-\beta z^t+\frac{\mu}{r}z^r+\BO\left(z^{t+\epsilon}\right)+\lambda z^{s+1}w^{-1}(1+\BO(z))	
		 \right|
\end{align}

Now we need to use the assumption that $t\leq r$ so that $z^r$  does not dominate $z^t$ in \eqref{wbound}.  
We will also use the assumption that $s+1-\mu>t+2\epsilon$ to help make $z^{s+1}$ smaller than $z^tw$.  
Let
\begin{equation}\label{betahat}
\hat{\beta}:=\begin{cases}\beta=\Delta, & \mbox{ if }t<r \\ \beta-\frac{\mu}{r}=\Delta-\frac{\epsilon}{r},& \mbox{ if }t=r.\end{cases}\end{equation}

When $t=r$ and $\re\Delta>0$, we chose $0<\epsilon<r\re\Delta$ so that $\re\hat{\beta}>0$.  
Since $[v]$ is transversally attracting (i.e., $\re\Delta>0$) in Theorem~\hyperref[thmA]{A}, we can now assume $\re\hat{\beta}>0$.  We rewrite \eqref{wboundb} as:
\begin{align}\label{wbound2}
\frac{|w_1|}{|z_1|^{\mu}} &
	=\frac{|w|}{|z|^{\mu}}
		\left|
		1-\hat{\beta} z^t+\BO\left(z^{t+\epsilon}\right)+\lambda z^{s+1}w^{-1}(1+\BO(z))	
		 \right|
\end{align}

If $w=0$, then $\frac{|w_1|}{|z_1|^{\mu}}=|\BO(z^{s+1-\mu})|<1$ for $s+1\geq\mu+\epsilon$ and so $(z_1,w_1)\in D$.  \\

If $\displaystyle w\neq 0$, then $|w|=|z|^\gamma$ for some $\gamma:=\gamma(z,w)>\mu$. We split this situation into two cases. \\

Case 1.  Suppose $\gamma\geq\mu+\epsilon$.  
Then we can use \eqref{wbound2} and that $s+1\geq\mu+\epsilon$ to get:
\begin{align*}
\frac{|w_1|}{|z_1|^{\mu}} &
	\leq |z|^{\epsilon}
		\left|
		1-\hat{\beta} z^t+\BO\left(z^{t+\epsilon}\right)\right|
		+\left|\lambda z^{s+1-\mu}(1+\BO(z))\right|
		<2|z|^\epsilon
		<1
\end{align*}
Then $(z_1,w_1)\in D$.  \\

Case 2.  Suppose $\mu<\gamma<\mu+\epsilon$.  
 Then we rewrite \eqref{wbound2} as:
\begin{align}\label{wbound3}
\frac{|w_1|}{|z_1|^{\mu}} &
	=\frac{|w|}{|z|^{\mu}}
		\left|
		1-\hat{\beta} z^t+z^t\BO\left(z^{\epsilon}, z^{s+1-\gamma-t}\right) \right|
		<\left|
		1-\hat{\beta} z^t+\BO\left(z^{t+\epsilon}\right) \right|<1,  
\end{align}
by using the bound $s+1\geq\mu+t+2\epsilon$.  Then $(z_1,w_1)\in D$.  \\

Therefore $D$ is $f$-invariant (i.e., $f(D)\subset D$).  Now we show that $(z_n,w_n)\to (0,0)$ along $[1:0]$. \\

For any $(z,w)\in D$, 
since
\begin{align*}
z_1&=z\left(1-\frac{1}{r}z^{r}+\BO\left(z^{r+\epsilon}\right)\right)
\end{align*}
we see, by using standard techniques, that:
\begin{equation}\label{znbdd}
\frac{1}{z_1^r}=\frac{1}{z^r}+1+\BO(z^\epsilon)\quad\Rightarrow\quad 
z_n\sim (z^{-r}+n)^{-\frac{1}{r}}
\quad\Rightarrow\quad z_n\sim\frac{1}{n^\frac{1}{r}}.
\end{equation}
Furthermore, for any $(z,w)\in D$,
$$|w_n|<|z_n|^{\mu}\sim\frac{1}{n^{\frac{\mu}{r}}}\quad\mbox{ and }\quad \frac{|w_n|}{|z_n|}<|z_n|^{\mu-1}\sim\frac{1}{n^\frac{\mu-1}{r}}.$$ 
Since $\mu>1$, we see that $z_n,w_n,\frac{w_n}{z_n}\to0$ as $n\to\infty$.  Hence, for any $(z,w)\in D$, $(z_n,w_n)$ converges to the origin along $[1:0]$.  
This completes the proof of the main theorem. \\

%
We now find more precise estimates on the size of $w_n$. 

\begin{prop}\label{size}Suppose $f$ satisfies the conditions of Theorem~\hyperref[thmA]{A} and express $f$ as in \eqref{feq3}.  Let $\nu:=-\frac{s+1-t}{r}+\re\beta$.  For any $(z,w)\in D$ and large $n\in\N$, $z_n\sim n^{-\frac{1}{r}}$ and: 
\begin{equation}\label{wnbdd}
w_n\sim\begin{cases} 
e^{-\frac{r\re\beta}{r-t}n^{\frac{r-t}{r}}}&\mbox{ if }t\neq r,s=\infty \\
n^{-\frac{s+1-t}{r}}&\mbox{ if }t\neq r,s\neq\infty\mbox{ or }t=r,\nu>0 \\
n^{-\re\beta}&\mbox{ if }t=r,\nu<0 \\
\displaystyle n^{-\re\beta}\log n   &\mbox{ if }t=r,\nu=0.
\end{cases}\end{equation}
\end{prop}\vspace{.5pc}

We use $w_n\sim g(n)$ to say that $w_n=\BO(g(n))$, so there may be  stricter bounds on $w_n$.

\begin{rem}\label{bddgamma}
If we let $\sim_i$ represent $<,>,$ or $=$, we see that $\nu\sim_i 0\Leftrightarrow \re\beta\sim_i \frac{s+1-t}{r}$.  \end{rem}

Our bounds in Proposition~\ref{size} are strictly better than our previous bound that $w_n\sim n^{-\frac{\mu}{r}}$:

\begin{itemize}
\item $s+1-t >\mu$ by Remark~\ref{smubdd}; 
\item $\re\beta>\frac{\mu}{r}$ when $t=r$ by \eqref{betahat}; and
\item $n^{-\re\beta}\log n<n^{-\frac{\mu}{r}}$ for large $n$ since $\re\beta>\frac{\mu}{r}$ when $t=r$.
\end{itemize}\vspace{0.5pc}

\begin{proof} 
%
We already showed that $z_n\sim n^{-\frac{1}{r}}$.  Now we bound $w_n$ by using \eqref{feq3}.  For any $n\in\N$, let: 
$$a_n:=1-\beta z_n^t+\BO\left(z_n^{t+\epsilon}\right)\quad\mbox{ and }\quad b_n:=z_n^{s+1}\lambda S(z_n)$$ 
such that $w_{n+1}=a_nw_n+b_n$ and $w:=w_0$.  Then:
\begin{align}\label{wn}
w_{n+1}
= a_n\left(a_{n-1} w_{n-1}+b_{n-1}\right)+b_n 
=
w\left(\prod_{j=0}^n a_j\right)+ \sum_{l=0}^n b_l\left( \prod_{j=l+1}^n a_j\right).
\end{align}
Let $\displaystyle A_{m,n}:=\prod_{j=m}^n (1-\beta z_j^t)$ for $m\leq n$ and $A_{m,n}=1$ for $m>n$.  Then \eqref{wn} becomes:
\begin{align} 
w_{n+1}&=
wA_{0,n}+ \lambda\sum_{l=0}^n z_l^{s+1}A_{l+1,n}+(h.o.t.)
\end{align}
We bound $A_{m,n}$ (for $0\leq m\leq n$):
\begin{align*}
\log A_{m,n}
&= \sum_{j=m}^n \log(1-\beta z_j^t)
=  -\beta \sum_{j=m}^n  z_j^t(1+\BO(z_j^{2t}))
\end{align*}
From \eqref{znbdd}, we know that $|z_j|<(j+2)^{-\frac{1}{r}}$ for $j>0$, so: 
$$-\beta^{-1}\log A_{m,n}\sim \sum_{j=m}^n (j+2)^{-\frac{t}{r}}\sim \int_{m-1}^n (x+2)^{-\frac{t}{r}}dx\sim \begin{cases} \frac{r}{r-t}\left((n+2)^{\frac{r-t}{r}}-(m+1)^{\frac{r-t}{r}}\right), & \mbox{ if }t\neq r \\ 
\log\left(\frac{n+2}{m+1}\right),&\mbox{ if }t=r,\end{cases}$$
where the sign of $\re(-\beta^{-1}\log A_{m,n})>0$.  Bounding $A_{0,n}$, the first term of $w_{n+1}$, we see that:
\begin{equation}\label{A0nbdd}
A_{0,n} \sim\begin{cases} 
e^{-\frac{r\re\beta}{r-t}(n+1)^\frac{r-t}{r}}
&\mbox{ if }t\neq r \\ (n+1)^{-\re\beta} &\mbox{ if }t=r.\end{cases}\end{equation}
If $s=\infty$, then $w_{n+1}=wA_{0,n}$ and so the previous estimates give us bounds on $w_{n+1}$.  If $s\neq\infty$, then we need to bound the second term of $w_{n+1}$.  Then:\begin{align}
\sum_{l=0}^n z_l^{s+1} A_{l+1,n} & \sim
\sum_{l=0}^n (l+2)^{-\frac{s+1}{r}} 
\begin{cases} e^{\frac{-r\re\beta}{r-t}\left[(n+2)^\frac{r-t}{r}-(l+2)^\frac{r-t}{r}\right]},
&\mbox{ if }t\neq r,s\neq\infty \\ \left(\frac{n+2}{l+2}\right)^{-\re\beta}, &\mbox{ if }t=r.\end{cases} \notag \\
&\sim
\begin{cases} 
\displaystyle\int_1^{n+2} x^{-\frac{s+1}{r}}e^{\frac{-r\re\beta}{r-t}\left[(n+2)^\frac{r-t}{r}-x^\frac{r-t}{r}\right]}dx,
& \mbox{ if }t\neq r,s\neq\infty \\  
\displaystyle (n+2)^{-\re\beta}\int_1^{n+2} x^{\nu-1} dx, &\mbox{ if }t=r.\end{cases} \notag \\  
&\sim
\begin{cases} 
\displaystyle e^{\frac{-r\re\beta}{r-t}(n+2)^\frac{r-t}{r} }\int_1^{n+2} x^{-\frac{s+1}{r}}e^{\frac{r\re\beta}{r-t}x^\frac{r-t}{r}}dx,
& \mbox{ if }t\neq r,s\neq\infty  \\  
\displaystyle (n+2)^{-\re\beta}\nu^{-1}\left( (n+2)^\nu-1\right), &\mbox{ if }t=r,\nu\neq0\\  
\displaystyle (n+2)^{-\re\beta}\log(n+2)   &\mbox{ if }t=r,\nu=0. 
\end{cases}\label{tneqreq}  \\
&\sim
\begin{cases} 
(n+1)^{-\frac{s+1-t}{r}},
& \mbox{ if }t\neq r,s\neq\infty\mbox{ (see Remark~\ref{tneqr})} \\  
(n+1)^{-\frac{s+1-t}{r}},
& \mbox{ if }t=r,\nu>0 \\  
\displaystyle (n+1)^{-\re\beta}, &\mbox{ if }t=r,\nu<0 \\ 
\displaystyle (n+1)^{-\re\beta}\log n   &\mbox{ if }t=r,\nu=0.  \label{sbdd}
\end{cases} 
\end{align}


\begin{rem}\label{tneqr}
We explain how to get from \eqref{tneqreq} to \eqref{sbdd} when $t\neq r$.  Let:
$$a=-\frac{s+1}{r}\quad b=\frac{r\re\beta}{r-t},\quad,c=\frac{r-t}{r}\quad\Rightarrow a-c+1=-\frac{s+1-t}{r}.$$
Then the integral in \eqref{tneqreq}, ignoring the constant in front, simplifies to: 
\begin{align}
\int_1^{n+2} x^a e^{b x^c}dx
&=(bc)^{-1} \int_1^{n+2} x^{a-c+1}  \left[ bc x^{c-1} \right]e^{b x^c}dx \notag \\
&=(bc)^{-1}\left[x^{a-c+1}e^{bx^c} \Big|_1^{n+2}
-\int_1^{n+2} (a-c+1) x^{a-c} e^{bx^c} dx\right] \notag \\
&= (bc)^{-1}\left( (n+2)^{-\frac{s+1-t}{r}}e^{b(n+2)^c}-e^b\right)
+\int_1^{n+2} (-a+c-1) x^{a-c} e^{bx^c} dx \label{bddint}
\end{align}
For large enough $x$, the integrand is strictly increasing.  We can ignore small values of $x$ (i.e., $1\leq x\leq \kappa \ll n$) in \eqref{bddint} since the constant multiple in front ($e^{-b (n+2)^c}$) will make the integral on $[1,\kappa]$ go to zero as $n$ goes to infinity.  So we bound the integral in \eqref{bddint} by the integral on $[\kappa,n]$:
$$\int_1^{n+2} (-a+c-1) x^{a-c} e^{bx^c} dx
\sim (n+2) (n+2)^{a-c} e^{b(n+2)^c}.$$
Therefore, when $t\neq r,s\neq\infty$, we bound \eqref{tneqreq} by \eqref{sbdd} as follows:
$$e^{-b (n+2)^c}\int_1^{n+2} x^a e^{b x^c}dx
\sim (n+2)^{-\frac{s+1-t}{r}}.$$
\end{rem}

Finally, we combine \eqref{A0nbdd} and \eqref{sbdd} to get the bounds in \eqref{wnbdd}.  
The bound in the case when $t=r,\nu>0$ follows by Remark~\ref{bddgamma}.  This concludes the proof.
\end{proof}

\section[Summary]{Summary of results on attracting domains in $\C^2$} \label{summary}

Let $f\in\TId(\C^2,\0)$ have characteristic direction $[v]$.  In the following table, we summarize results on the existence of invariant attracting domains whose points converge to $\0$.  In order to simplify the table, let: \begin{align*}
A \quad& \mbox{ be Abate's Index of }f\mbox{ at }[v] \mbox{ (see \ref{director})}. \\
\Delta \quad& \mbox{ be the director of }f\mbox{ corresponding to }[v]
\mbox{ (see \ref{director}).}\\
R\quad & \mbox{ be }\left\{z\in\C ~\bigg|~\re(z)>-\frac{m}{k}, \bigg|z-\frac{m+1-\frac{m}{k}}{2}\bigg|>\frac{m+1+\frac{m}{k}}{2}\right\}\subset\C,\\ 
U \quad& \mbox{ be an open set whose points converge to }\0,\mbox{ each point along some direction.} \\
\Oa \quad& \mbox{ be an invariant domain of attraction whose points converge to }\0\mbox{ along }[v].  \\
 \Ob \quad& \mbox{ be }\Oa\mbox{ such that }f\mbox{ is conjugate to translation on }\Oa. \\
 \Oc \quad& \mbox{ be }\Ob\mbox{ and a Fatou-Bieberbach domain when }f\mbox{ is an automorphism}.  \\
\F \quad& \mbox{ be }\{f \ | \ f,[v]\mbox{ satisfy assumptions of Theorem }\hyperref[thmA]{A}\},\mbox{ which implies } \exists\Oa.
\end{align*}
\vspace{-1pc}
\renewcommand{\tabcolsep}{2pt}
 \renewcommand\arraystretch{1.4}

 \begin{savenotes}
\begin{table}[h]
\begin{center}\begin{tabular}{|l|lll|lll|}
\hline
\textbf{$[v]$ is} & \textbf{Non-Degenerate} & & $m=0$ & \textbf{Degenerate} &  &$m>0$ \\
& \textit{all cases known} & \textit{ for }$\Oa$&&$\exists$\textit{ open cases}&\textit{ for }$\Oa$&\\
\hline
\textbf{Fuchsian} & $\bullet \re \Delta> 0\quad\Rightarrow $ & $\left\{ \begin{array}{l}
 	\exists\Oc \\
	\exists (k-1)~\Oa
	\end{array}\right.$ 
 & $ \begin{array}{l}
 	\mbox{[H1]} \\
	\mbox{[AR]}
	\end{array}$  
& $\bullet A\in R$ & $\Rightarrow\exists \Oa$  & \cite{V2}\\ 
\textbf{$1+m=n$} & $\bullet \re \Delta=0\neq \Delta$ & $\Rightarrow\nexists \Oa$  &   \cite{L1,Ro1} & $\bullet \exists A\not\in R$   & s.t. $\nexists\Oa$ & \cite{L1}\footnote{The example given that satisfies these conditions was $f(z,w)=(z+w^2,w)$.} \\ 
 &$\bullet \re \Delta< 0$ &  $\Rightarrow \nexists\Oa$ &\cite{H1}&$\bullet f\in\F$&$\Rightarrow\exists\Oa$&Thm.~\hyperref[thmA]{A}\\
 \hline
 $\begin{array}{l}
 	\textbf{Irregular}  \\
	1+m<n 	\end{array}$ 
	 & $\bullet\left. \begin{array}{l}
	\Delta\equiv 0 \\
	r\not\equiv0 \\
	\end{array}\right\}$ always
& $\Rightarrow \exists \Oc$ & ~\cite{L2,V2}
	 & $\bullet$
	always
& $\Rightarrow \exists \Oc$ & ~\cite{V2}  \\

 \hline
 $\begin{array}{l}
 	\textbf{Apparent}  \\
	1+m>n 	\end{array}$    
& Does not apply   && & $\bullet\exists$ examples & s.t.~$\left\{ \begin{array}{l}
 	\exists\Oa \\
	\nexists\Oa
	\end{array}\right.$ 
&  $\begin{array}{l}
 	\mbox{[V2]}  \\ 
	\mbox{[V2]} 	\end{array}$  \\ 
	&&&& $\bullet f\in\F$&$\Rightarrow\exists\Oa$&Thm.~\hyperref[thmA]{A},\cite{Ro2}\footnote{The main theorem in \cite{Ro2} applied only to apparent characteristic directions and more specifically required $t=k$, whereas Theorem~\hyperref[thmA]{A} applies to apparent (and other) characteristic directions for $t\geq k$.}\\
 \hline
$ \begin{array}{l}
 	\0\mbox{ is}\\
	\mbox{\textbf{Dicritical}} \end{array}$
 & $\bullet\left. \begin{array}{l}
	\Delta\equiv 0 \\
	r\equiv 0
	\end{array}\right\}$ always &$\Rightarrow
\left\{ \begin{array}{l}
 	\exists U \\
	\nexists \Oa
	\end{array}\right.$ 
 & $ \begin{array}{l}
 	\mbox{[Bro]} \\
	\mbox{[Ro1]}
	\end{array}$   
	& 
$\bullet$ always
&  $\Rightarrow \exists U$ & \cite{Bro}  \\
~$n=\infty$ &&&& $\bullet f\in\F$&$\Rightarrow\exists\Oa$&Thm.~\hyperref[thmA]{A}\\
\hline
 \end{tabular}\caption{Summary of the existence of invariant domains of attraction in $\C^2$.}\label{C2table} \end{center}\end{table}
 \end{savenotes}

As we see from Table~\ref{C2table}, Theorem~\hyperref[thmA]{A} shows new instances in which there exists a domain of attraction along $[v]$ (i.e., $\exists\Oa$).  In particular, the new instances correspond to when $k<t\leq r$, so that $[v]$ is degenerate.  In this case, Theorem~\hyperref[thmA]{A} proves the existence of a domain of attraction along $[v]$ (when its conditions are satisfied) for all four categories of degenerate characteristic directions: Fuchsian, irregular, apparent, and dicritical. \\

\end{document}